\documentclass{amsart}
\usepackage{graphicx} % Required for inserting images
\usepackage{amsfonts,amsthm,mathtools,amsmath,amssymb,enumitem,xcolor,fullpage,url,hyperref}
\usepackage{tikz}
\newtheorem{theorem}{Theorem}
\newtheorem{proposition}{Proposition}
\newtheorem{lemma}{Lemma}
\newtheorem{corollary}{Corollary}

\theoremstyle{remark}

\usepackage{tabularx}
\usepackage{booktabs}
\usepackage{float}

\newcommand{\Aut}{\operatorname{Aut}}
\newcommand{\Gal}{\operatorname{Gal}}
\newcommand{\Jac}{\operatorname{Jac}}

\newcommand{\Q}{\mathbb{Q}}

\newcommand{\Z}{\mathbb{Z}}
\newcommand{\C}{\mathbb{C}}
\newcommand{\GL}{\operatorname{GL}}

\newcommand{\SL}{\operatorname{SL}}

\newcommand{\PP}{{\mathbb P}}
\newcommand{\im}{\operatorname{im}}

\title{Rational Isolated $j$-invariants from $X_1(\ell^n)$ and $X_0(\ell^n)$}

\author{Abbey Bourdon}
\address{Wake Forest University, Winston-Salem, NC 27104, USA}
\email{bourdoam@wfu.edu}
\author{\"{O}zlem Ejder}
\address{Koc University, Istanbul, Turkey}
\email{ozejder@ku.edu.tr}
\begin{document}
\begin{abstract}
    Let $\ell$ and $n$ be positive integers with $\ell$ prime. The modular curves $X_1(\ell^n)$ and $X_0(\ell^n)$ are algebraic curves over $\Q$ whose non-cuspidal points parameterize elliptic curves with a distinguished point of order $\ell^n$ or a distinguished cyclic subgroup of order $\ell^n$, respectively. We wish to understand isolated points on these curves, which are roughly those not belonging to an infinite parameterized family of points having the same degree. Our first main result is that there are precisely 15 $j$-invariants in $\Q$ which arise as the image of an isolated point $x\in X_1(\ell^n)$ under the natural map $j:X_1(\ell^n) \rightarrow X_1(1)$. This completes a prior partial classification of Ejder. We also identify the 19 rational $j$-invariants which correspond to isolated points on $X_0(\ell^n)$.
\end{abstract}

\maketitle

\section{Introduction}
Let $C$ be a nice curve\footnote{By ``nice" we mean smooth, projective, and geometrically integral.} defined over a number field $k$. We say a degree $d$ closed point $x \in C$ is \textbf{isolated} if it is not part of an infinite family of degree $d$ points parameterized by the projective line or a positive rank abelian subvariety of the curve's Jacobian (see Section~\ref{sec:isolated} for details). A special class of isolated points are \textbf{sporadic} points: those $x \in C$ for which there are only finitely many points of degree at most $\deg(x)$. Sporadic points on the modular curve $X_1(n)$ have been of interest for years, partly owing to their importance in classifying torsion subgroups of elliptic curves over number fields of fixed degree. Recall that $X_1(n)$ is a curve over $\Q$ whose (non-cuspidal) points parameterize elliptic curves with a distinguished point of order $n$. Early examples of sporadic points on $X_1(n)$ include those found by van Hoeij \cite{vanHoeij} and Najman \cite{najman16}, who identified sporadic points of degree 6 on $X_1(37)$ and sporadic points of degree 3 on $X_1(21)$, respectively. However, one of the first papers to propose a systematic study of isolated points on $X_1(n)$ is that of the present authors in collaboration with Liu, Odumodu, and Viray \cite{BELOV}. There, the authors show that a positive answer to Serre's Uniformity Problem would imply there are only finitely many $j\in\Q$ which are the image of an isolated point $x \in X_1(n)$ under the natural map $j:X_1(n) \rightarrow X_1(1)$, even as $n$ ranges over all positive integers. This is the set of so-called \textbf{isolated $j$-invariants} in $\Q$, with respect to the family $X_1(n)$. We will call such $j$-invariants \textbf{$\Gamma_1$-isolated}. Note that $j\in X_1(1) \cong \mathbb{P}^1$ is never isolated as a point on $X_1(1)$, so the terminology is only a convenient way to characterize elliptic curves (up to $\overline{\Q}$-isomorphism) giving rise to an isolated point on $X_1(n)$ for some $n \in \Z^+$. This is in analogy to how we use the term ``singular moduli" for $j$-invariants of CM elliptic curves. By \cite[Theorem 7.1]{BELOV}, all singular moduli are $\Gamma_1$-isolated.

Several papers provide partial progress towards classifying the set of rational $\Gamma_1$-isolated $j$-invariants \cite{BELOV,Ejder22,OddDeg}, which has recently been conjectured to have only 17 elements; see \cite{Algorithm2025}. In particular, it is known to contain every $j$-invariant associated to an elliptic curve with complex multiplication (CM), along with 4 non-CM $j$-invariants. In the present work, we provide an unconditional classification of rational $\Gamma_1$-isolated $j$-invariants associated to $X_1(\ell^n)$ where $\ell$ is prime. This completes work of the second author \cite{Ejder22} for $\ell>7$. Here $j(x)$ denotes the image of $x\in X_1(n)$ under the $j$-map $X_1(n) \rightarrow X_1(1)$.

\begin{theorem}\label{X1ellThm}
    Let $\ell$ be prime and let $n$ be a positive integer. Then $j\in\Q$ arises as $j=j(x)$ for an isolated point $x\in X_1(\ell^n)$ if and only if $j$ is a CM $j$-invariant, $-7\cdot 11^3$, or $-7 \cdot 137^3\cdot 2083^3$. Moreover, the non-CM $j$-invariants occur if and only if $\ell=37$.
\end{theorem}

The main result of \cite{Ejder22} relies on work of Lozano-Robledo \cite[Theorem 1.2]{LR16}, which contains an error (see \cite[$\S 1$]{Smith23} for details). Nevertheless, the argument remains valid by instead appealing to work of Smith \cite{Smith23}. We provide a corrected proof in Proposition~\ref{PropDegModCurve}. It thus remains to give a classification for primes $\ell \leq 7$. Our main technique is \cite[Theorem 4.3]{BELOV}, which concerns the image of isolated points belonging to irreducible fibers. We use complete and partial classification results on the image of $\ell$-adic Galois representations associated to elliptic curves over $\Q$ to show that a hypothetical isolated point $x \in X_1(\ell^n)$ would map down to an isolated point $y\in X_1(\ell^k)$ for a small explicit $k \in \Z^+$. Aside from the stated exceptions, we can show $y$ is not isolated, meaning the isolated point $x$ does not exist.

Similar tools can be used to address isolated points on the modular curve $X_0(n)$, which have previously been studied by Menendez \cite{Menendez}, Terao \cite{Terao}, and Lee \cite{Lee25}. Recall $X_0(n)$ is an algebraic curve over $\Q$ whose (non-cuspidal) points parameterize elliptic curves with a distinguished cyclic subgroup of order $n$. If $x \in X_0(n)$ is isolated, we say $j(x)\in X_1(1)$ is \textbf{$\Gamma_0$-isolated}. By \cite[Theorem 44]{Lee25}, all singular moduli are $\Gamma_0$-isolated, and there are 14 known non-CM $j$-invariants in $\Q$ which are $\Gamma_0$-isolated (see \cite{Terao,Lee25}). Our second main result gives an unconditional classification of rational $\Gamma_0$-isolated $j$-invariants associated to curves of prime-power level. As before, $j(x)$ denotes the image of a point under the $j$-map $X_0(\ell^n) \rightarrow X_0(1)$.

\begin{theorem}\label{X0ellThm}
        Let $\ell$ be prime and let $n$ be a positive integer. Then $j\in\Q$ arises as $j=j(x)$ for an isolated point $x\in X_0(\ell^n)$ if and only if $j$ is a CM $j$-invariant, $-11\cdot 131^3$, $-11^2$, $-17^2\cdot 101^3/2$, $-17\cdot 373^3/2^{17}$, $-7\cdot 11^3$, or $-7 \cdot 137^3\cdot 2083^3$. Moreover, the non-CM $j$-invariants in this list correspond to isolated rational points on $X_0(\ell)$.
\end{theorem}

We note that the set of rational $\Gamma_1$-isolated $j$-invariants in Theorem 1 is contained in the set of rational $\Gamma_0$-isolated $j$-invariants given in Theorem 2. However, this containment fails to hold away from prime-power level: by \cite[Theorem 1]{Lee25}, the $j$-invariant $351/4$ is $\Gamma_1$-isolated but not $\Gamma_0$-isolated. 

\subsection{Outline}
In Section 3 we use a modified version of the algorithms of \cite{Algorithm2025,Lee25} to characterize isolated points on $X_1(\ell^n)$ and $X_0(\ell^n)$ corresponding to non-CM elliptic curves over $\Q$ with $\ell$-adic Galois representation appearing in classifications of Rouse and Zureick-Brown \cite{2adicImage} or Rouse, Sutherland, and Zureick-Brown \cite{RSZB2022}. If $E/\Q$ is a non-CM elliptic curve whose mod $\ell$ Galois representation has image in the normalizer of a non-split Cartan subgroup, then recent work of Furio \cite{Furio2024} characterizes the possible full $\ell$-adic image of Galois. This partial classification is enough to rule out isolated points on $X_1(\ell^n)$ or $X_0(\ell^n)$ associated to $E$, which we show in Section 4. This leaves only the (possibly non-existent) case of non-CM elliptic curves $E/\Q$ with 7-adic image of Galois contained in an exceptional subgroup labeled 49.196.9.1. In Section 5, we prove the following proposition, which allows us to complete the proofs of our two main results.

\begin{proposition}
    Suppose $E/\Q$ is a non-CM elliptic curve with $\im \rho_{E,7^{\infty}}$ contained in $49.196.9.1$. Then  $\im \rho_{E,7^{\infty}}=49.196.9.1$.
\end{proposition}

\subsection{Code}
We make frequent use of the computer algebra system Magma \cite{Magma}. All code and relevant data is available at
\url{https://github.com/abbey-bourdon/rational-isolated-prime-power}.

\section*{Acknowledgments}
The first author was partially supported by NSF grant DMS-2145270 and the second author was partially supported by TUBITAK-124F203 grant. We thank Jeremy Rouse for useful discussions concerning the data files associated to \cite{2adicImage} and \cite{RSZB2022}, including his creation of the file \href{https://github.com/abbey-bourdon/rational-isolated-prime-power/blob/main/elladicgens.txt}{\tt{elladicgens.txt}} which contains a list of all known $\ell$-adic images associated to non-CM elliptic curves over $\Q$.

\section{Background}
\subsection{Conventions} Throughout $k$ denotes a number field and $\Gal_k$ is used to denote the absolute Galois group of $k$. We use $\ell$ to denote a prime number and $\Z_{\ell}$ for the ring of $\ell$-adic integers.

By curve we mean a projective nonsingular $1$-dimensional scheme over $k$. For a curve $C/k$, we define the degree of a closed point $x\in C$ to be the degree of its residue field $k(x)$ as an extension of $k$.

  An elliptic curve is a curve of genus $1$ with a specified base point. We say an elliptic curve $E$ over $k$ has complex multiplication, or {CM}, if its geometric endomorphism ring is strictly larger than $\mathbb{Z}$. For $P \in E(\overline{k})$, we use $k(P)$ to denote the field extension of $k$ generated by the $x$- and $y$-coordinates of $P$.

  For any open subgroup $G \leq \GL_2(\Z_\ell)$, we define its level to be the smallest integer $\ell^n$ such that $G=\pi_n^{-1}(G(\ell^n))$, where $G(\ell^n)$ is the image of $G$ under the projection map $\pi_n: \GL_2(\Z_{\ell}) \to  \GL_2(\Z/\ell^n \Z)$. 

  Subgroups of $\GL_2(\Z/\ell\Z)$ appearing as the image of the mod $\ell$ Galois representation of an elliptic curve over $\Q$ are described using the notation of Sutherland \cite{sutherland}. Subgroups of $\GL_2(\Z_{\ell})$ appearing as the image of the $\ell$-adic Galois representation of a non-CM elliptic curve over $\Q$ are denoted by the notation of Rouse, Sutherland, and Zureick-Brown \cite{RSZB2022}. For an open subgroup $G\leq \GL_2(\Z_{\ell})$, this notation is of the form
  \begin{center}
      \tt{N.i.g.n}
  \end{center}
  where $N$ denotes the level, $i$ denotes the index, $g$ denotes the genus of the modular curve $X_G$, and $n$ is an integer used to distinguish groups having the same level, index, and genus.
 
 \subsection{Galois Representations}
Let $E/k$ be an elliptic curve and let $\ell$ be a prime number. By fixing a basis for the $\ell$-adic Tate module of $E$, we obtain the $\ell$-adic representation given by  
\[ \rho_{E,\ell^{\infty}}: \Gal_k \to \GL_2(\Z_{\ell}).
\]
This describes the Galois action on all points in $E(\overline{k})$ with order a power of $\ell$. Similarly for any integer $n \geq 1$, we also have the mod $\ell^n$ representation
 \[ \rho_{E,\ell^n}: \Gal_k \to \GL_2(\Z/{\ell^n}\Z),
\]
which describes the Galois action on the $\ell^n$-torsion subgroup $E[\ell^n]$ of $E(\bar{k})$. We denote the image of $ \rho_{E,\ell^{\infty}}$ and $ \rho_{E,\ell^n}$ by $\im \rho_{E,\ell^{\infty}}$ and $\im \rho_{E,\ell^n}$, respectively. By Serre's Open Image Theorem \cite{serre72}, we know that $\im \rho_{E,\ell^{\infty}}$ is open in $\GL_2(\Z_{\ell})$ for any non-CM elliptic curve $E/k$, so its level is well defined.
  \subsection{Modular Curve $X_1(n)$} For a positive integer $n$, the modular curve $X_1(n)$ is an algebraic curve over $\Q$ whose non-cuspidal points correspond to isomorphism classes of elliptic curves with a distinguished point of order $n$. We may understand the complex points of $X_1(n)$ by first constructing $\mathbb{H}/\Gamma_1(n)$, where $\mathbb{H}$ is the upper half plane and the group
        \[
            \Gamma_1(n) \coloneqq
            \left\{{\left(\begin{smallmatrix}a & b \\ c & d \end{smallmatrix}\right)} \in \SL_2(\Z) :  c \equiv 0 \pmod{n}, \, a \equiv d \equiv 1 \pmod{n}\right\} 
        \]
        acts via linear fractional transformations. By adding a finite number of cusps, we obtain a compact Riemann surface whose complex points agree with $X_1(n)(\C)$. The non-cuspidal points of $X_1(n)(\C)$ correspond to pairs $(E,P)$ where $E/\C$ is an elliptic curve and $P\in E(\C)$ is of order $n$, up to $\C$-isomorphism. If $n \geq 4$, the same holds true when $\C$ is replaced by any number field $k$. That is, for these $n$ the curve $X_1(n)$ is a fine moduli space.

If $E/k$ is an elliptic curve with $P\in E(\overline{k})$ of order $n$, then $(E,P)$ induces a $k$-valued point on $X_1(n)$ by the moduli interpretation. However, as explained in $\S2.1$ we define the degree of $(E,P)$ to be the degree of the residue field of the corresponding closed point $[E,P]$. This degree can be computed explicitly using the following result.

\begin{lemma}\label{ResidueFieldLemma}
Let $E$ be an elliptic curve defined over $\Q(j(E))$, and let $P \in E$ be a point of order $n$. Then the residue field of the closed point $x \in X_1(n)$ associated to $(E,P)$ is given by
\[
\Q(x)\cong\Q(j(E),\mathfrak{h}(P)),
\]
where $\mathfrak{h} \rightarrow E/\Aut(E) \cong \mathbf{P}^1$ is a Weber function for $E$.
\end{lemma}

\begin{proof}
See, for example, \cite[Lemma 2.5]{OddDegQCurve}.
\end{proof}

Our study of isolated points makes frequent use of the natural maps between modular curves. We record the degree of one such map below.
    \begin{proposition}\label{prop:degree}
                For positive integers $a$ and $b$, there is a natural $\Q$-rational map $f \colon X_1(ab) \rightarrow X_1(a)$ sending $[E,P]$ to $[E, bP]$ with
                \[
                    \deg(f)=
                    c_f\cdot b^2 \prod_{p \mid b, p \nmid a}\left( 1-\frac{1}{p^2} \right).             \]
                    Here, $c_f=1/2$ if $a \leq 2$ and $ab>2$ and $c_f=1$ otherwise.
          \end{proposition}

          \begin{proof}
              The degree computation follows from \cite[p.66]{modular}.
          \end{proof}

\subsection{Modular Curve $X_0(n)$} The modular curve $X_0(n)$ is an algebraic curve over $\Q$ whose non-cuspidal points correspond to isomorphism classes of elliptic curves with a distinguised cyclic subgroup of order $n$. As above, we may understand $X_0(n)(\C)$ as the compactification of the Riemann surface given by $\mathbb{H}/\Gamma_0(n)$, where
        \[
            \Gamma_0(n) \coloneqq
            \left\{{\left(\begin{smallmatrix}a & b \\ c & d \end{smallmatrix}\right)} \in \SL_2(\Z) :  c \equiv 0 \pmod{n}\right\} 
        \]
acts by linear fractional transformations. If $E/k$ is an elliptic curve and $C$ is a $k$-rational cyclic subgroup of order $n$, then $(E,C)$ gives a point in $X_0(n)(k)$. However, in general these pairs are equivalent if they agree up to $\overline{k}$-isomorphism. That is, $X_0(n)$ is not a fine moduli space.

For a non-CM elliptic curve $E/\Q(j(E))$ with a cyclic subgroup $C$ of order $n$, the residue field of the closed point $[E,C]$ is precisely the extension over $\Q(j(E))$ over which $C$ is fixed (as a set) by the action of Galois. See \cite[$\S 3.3$]{ClarkVolcanoes} for details. The degree of the natural map between modular curves of the form $X_0(n)$ is given below.

\begin{proposition}
 For positive integers $a$ and $b$, there is a natural $\Q$-rational map $f:X_0(ab) \rightarrow X_0(a)$ sending $[E, \langle P \rangle]$ to $[E, \langle bP \rangle]$ with
        \begin{align*} \deg(f) =
            b \prod_{p \mid b, p \nmid a}\left(1 + \frac{1}{p}\right).
        \end{align*}

    \begin{proof}
        This follows from the formulas of \cite[p. 66]{modular}.
    \end{proof}
\end{proposition}
\subsection{Isolated Points}\label{sec:isolated}

Let $C/k$ be a curve with a point $P \in C(k)$. For a positive integer $d$, let $C^d$ denote the direct product of $d$ copies of $C$. We denote the $d$-th symmetric product of $C$ by $C^{(d)}$, i.e., the quotient of $C^d$ by the symmetric group $S_d$.  We have a natural map
 $\phi_d \colon C^{(d)} \to \Jac(C)$ given by $(P_1,\ldots, P_d) \mapsto [P_1+\ldots P_d-dP]$. Let $x\in C$ be a closed point of degree $d$, which we may identify with a point in $C^{(d)}(k)$ by viewing it as a Galois orbit of points in $C(\overline{k})$ having length $d$. We say $x$ is:
 \begin{itemize} 
 \item\textbf{$\PP^1$-parameterized} if there is $x' \in C^{(d)}(k)$ such that $\phi_d(x)=\phi_d(x')$. Otherwise we say $x$ is \textbf{$\PP^1$-isolated}.
 \item \textbf{AV-parameterized} if there exists a positive rank subabelian variety $A$ of $\Jac(C)$ such that $\phi_d(x) + A \subset \im(\phi_d)$. Otherwise we say $x$ is \textbf{AV-isolated}.
 \end{itemize}
We say $x$ is \textbf{isolated} if it is both $\PP^1$-isolated and AV-isolated.

Any curve $C/k$ contains only finitely many isolated points, and it has infinitely many points of degree $d$ if and only if there exists a degree $d$ point which is $\PP^1$-parameterized or AV-parameterized; see \cite[Theorem 4.2]{BELOV}. In particular, if $C$ has genus at least 2, then any point in $C(k)$ is isolated by Faltings' Theorem \cite{faltings83}. Thus one can consider the study of isolated points as a kind of generalization of the study of rational points. See recent work of Viray and Vogt \cite{VirayVogt} for more discussion.

One key result that can be used to rule out the existence of isolated points is the following.

\begin{theorem}[Bourdon, Ejder, Liu, Odumodu, Viray \cite{BELOV}]\label{BELOV_level}
    Let $f:C \rightarrow D$ be a finite map of curves over $k$, and let $x \in C$ be an isolated point. If $\deg(x)=\deg(f(x))\cdot \deg(f),$ then $f(x) \in D$ is isolated.
\end{theorem}

\subsection{Prior work of Ejder} The following proposition essentially follows from work of the second author \cite[Proposition 3.1]{Ejder22}. However, the proof relies on \cite[Theorem 1.2]{LR16}, which contains an error: the statement of \cite[Lemma 5.3]{LR16} is incorrect and should instead be as in \cite[$\S4$]{Smith23}. For the sake of completeness, we include a corrected proof of \cite[Proposition 3.1]{Ejder22} which employs work of Smith \cite{Smith23} in place of \cite[Theorem 1.2]{LR16}. This has some overlap with \cite[Proposition 4]{BourdonGenao}.
\begin{proposition}[Ejder, \cite{Ejder22}]\label{PropDegModCurve}
Let $E/\Q$ be a non-CM elliptic curve. Suppose $\im \rho_{E,\ell}$ is equal to the normalizer of a nonsplit Cartan subgroup for some prime $\ell>13$. For $P \in E(\overline{\Q})$ of order $\ell^k$, let $x=[E,P] \in X_1(\ell^k)$ be the associated closed point on the modular curve. Then
\[
\deg(x)=\frac{1}{2}(\ell^2-1)\ell^{2k-2}=\deg(X_1(\ell^k) \rightarrow X_1(1)).
\]
\end{proposition}

\begin{proof}
    If $n=1$, this follows from \cite[Lemma 5.2]{GJNajman20}, so it remains to show $[\Q(P):\Q(\ell^{k-1} P)]=\ell^{2k-2}$ for $k>1$. By \cite[Proposition 3.1]{Ejder22}, the curve $E/\Q$ has potential good supersingular reduction at $\ell$. By \cite[$\S5.6$]{serre72}, $E$ attains good supersingular reduction over an extension $K/\Q_{\ell}^{nr}$ of degree at most 6, and $E$ has no canonical subgroup of order $\ell$ by \cite[Theorem 4.5]{Furio2024}. Ramification bounds of Smith \cite[Theorem 1.1]{Smith23} imply
\[
\ell^{2k}-\ell^{2k-2}
\]
divides the degree of $K(P)/\Q_{\ell}^{nr}$. As $\ell>13$, it follows that
\[
\ell^{2k-2} \mid [\Q(P):\Q].
\]
Since $[\Q(\ell^{k-1}P):\Q]$ is prime to $\ell$, the result follows.
\end{proof}

\begin{corollary}\label{CorModCurve}
 Let $E/\Q$ be a non-CM elliptic curve. Suppose $\im \rho_{E,\ell}$ is equal to the normalizer of a nonsplit Cartan subgroup for some prime $\ell>13$. For $P \in E(\overline{\Q})$ of order $\ell^k$, let $x=[E,\langle P \rangle] \in X_0(\ell^k)$ be the associated closed point on the modular curve. Then
\[
\deg(x)=(\ell+1)\ell^{k-1}=\deg(X_0(\ell^k) \rightarrow X_0(1)).
\]
\end{corollary}

\section{Isolated points from known images}
The following theorems characterize both $\Gamma_1$-isolated and $\Gamma_0$-isolated $j$-invariants in $\Q$ associated to known $\ell$-adic images, as given in \cite{2adicImage,RSZB2022}. They can be deduced from a slightly modified version of the algorithms used in \cite{Algorithm2025} and \cite{Lee25}. For example, the algorithm of \cite{Algorithm2025} was designed to take as input a given $j$-invariant, and then output a set of level, degree pairs $\langle a_i, d_i \rangle$ such that $j$ corresponds to an isolated point on $X_1(n)$ if and only if there exists an isolated point $x\in X_1(a_i)$ of degree $d_i$ with $j(x)=j$. In particular, if the output is the empty set, then there is no isolated point on $X_1(n)$ associated to $j$, even as $n$ ranges over all positive integers. The algorithm of \cite{Lee25} worked in an analogous way for the modular curve $X_0(n)$. By modifying these algorithms to take as input a given $\ell$-adic image, we may obtain the desired results.
\subsection{$\Gamma_1$-isolated $j$-invariants} \begin{theorem}\label{AlgThmGamma1}
    Let $E/\Q$ be a non-CM elliptic curve, and suppose $\im \rho_{E,\ell^{\infty}}$ is conjugate to one of the images known to exist for a non-CM elliptic curve as classified in \cite{2adicImage,RSZB2022}. Then there is no isolated point $x \in X_1(\ell^n)$ with $j(x)=j(E)$ unless $\ell=37$ and $j(x)=-7\cdot 11^3$ or $-7 \cdot 137^3\cdot 2083^3$. These $j$-invariants give rise to isolated points on $X_1(37)$ of degrees 6 and 18, respectively.
\end{theorem}

\begin{proof}
Given $\im \rho_{E,\ell^{\infty}}$, we will apply a streamlined version of the algorithm of \cite{Algorithm2025}. The output is a finite list of pairs $\{\langle \ell^{a_1}, d_1 \rangle, \langle \ell^{a_2}, d_2 \rangle, \dots, \langle \ell^{a_r}, d_r \rangle\}$ such that there exists an isolated $x\in X_1(\ell^n)$ with $j(x)=j(E)$ if and only if there is an isolated point $x_i \in X_1(\ell^{a_i})$ of degree $d_i$ with $j(x_i)=j(E)$, for some $i$. We outline the steps below. The implementation is given in \href{https://github.com/abbey-bourdon/rational-isolated-prime-power/blob/main/isolated_points_from_image_Gamma1.m}{\tt{isolated\_points\_from\_image\_Gamma1.m}}.
\begin{enumerate}
    \item For each $k \in \Z^+$ and each closed point $x \in X_1(\ell^k)$ associated to $E$, we use degree information from $\im \rho_{E,\ell^{\infty}}$ to determine the smallest $\ell^a$ for which $\deg(x)=\deg(f(x))\cdot \deg(f)$ where $f: X_1(\ell^k) \rightarrow X_1(\ell^a)$ is the natural map. If $x$ is isolated, then $f(x)$ is isolated by Theorem \ref{BELOV_level}. Since $\ell^a$ is bounded by the level of $\im \rho_{E,\ell^{\infty}}$, this gives a finite list of pairs $\langle \ell^{a_i}, d_i \rangle$ such that $E$ gives rise to an isolated point on $X_1(\ell^n)$ for some $n$ only if there exists an isolated point $x_i \in X_1(\ell^{a_i})$ with $j(x_i)=j(E)$ and $\deg(x_i)=d_i$ for some $i$.
    \item We remove any pairs $\langle \ell^{a_i}, d_i \rangle$ for which $d_i>\text{genus}(X_1(\ell^{a_i}))$. Such a point is $\mathbb{P}^1$-parameterized as its Riemann-Roch space has dimension at least 2 by the Riemann-Roch Theorem.
    \item We remove any pairs $\langle \ell^{a_i}, d_i \rangle$ for which $\im \rho_{E,\ell^{a_i}}$ corresponds to a modular curve of genus 0. Such a point is $\PP^1$-parameterized by \cite[Theorem 38]{Algorithm2025}.
\end{enumerate}

In particular, if the algorithm outputs the empty set, then there are no isolated points on $X_1(\ell^n)$ associated to an elliptic curve with the given $\ell$-adic image. We ran this algorithm on all known $\ell$-adic images associated to non-CM elliptic curves over $\Q$, as classified in \cite{2adicImage,RSZB2022}. These are given in the file \href{https://github.com/abbey-bourdon/rational-isolated-prime-power/blob/main/elladicgens.txt}{\tt{elladicgens.txt}}. The \href{https://github.com/abbey-bourdon/rational-isolated-prime-power/blob/main/Gamma1results.out}{\tt{output}} is the empty set unless $\im \rho_{E,\ell^{\infty}}=17.72.1.2$, $37.114.4.1$, or $37.114.4.2$. The image 17.72.1.2 outputs a set containing only $\langle 17, 4 \rangle$. Thus, an elliptic curve with this image yields an isolated point on $X_1(17^n)$ for some $n$ if and only if it gives rise to an isolated point of degree 4 on $X_1(17)$. However, any degree 4 point on $X_1(17)$ is $\mathbb{P}^1$-parameterized by \cite[Proposition 6.7]{DerickxMazurKamienny}. The images $37.114.4.1$ and $37.114.4.2$ occur only if $j(E)=-7\cdot 11^3$ or $-7 \cdot 137^3\cdot 2083^3$ and correspond to rational points on $X_0(37)$. These $j$-invariants appear in the theorem statement. Both are known to be isolated, as we will now explain. The first gives rise to a degree 6 point on $X_1(37)$. This is isolated by work of Frey \cite{frey}, since 6 is less than half the $\Q$-gonality of $X_1(37)$, as computed in \cite{DerickxVanHoeij2014}. The second is isolated by \cite[Theorem 2 \& 48]{Algorithm2025}.
\end{proof}

\subsection{$\Gamma_0$-isolated $j$-invariants} 

\begin{theorem}\label{AlgThmGamma0}
    Let $E/\Q$ be a non-CM elliptic curve, and suppose $\im \rho_{E,\ell^{\infty}}$ is conjugate to one of the images known to exist for a non-CM elliptic curve as classified in \cite{2adicImage,RSZB2022}. Then there is no isolated point $x \in X_0(\ell^n)$ with $j(x)=j(E)$ unless one of the following holds:
    \begin{table}[H]
        \centering
        \begin{tabular}{c|c}
        $\ell$ & j(x) \\
        \hline
           $11$  & $-11\cdot 131^3 \text{ \emph{or} }-11^2$ \\
           \hline
            $17$ &  $-17^2\cdot 101^3/2$ \text{ \emph{or} } $-17\cdot 373^3/2^{17}$\\
               \hline
            $37$ &  $-7\cdot 11^3$ \text{ \emph{or} } $-7 \cdot 137^3\cdot 2083^3$
        \end{tabular}
      %  \caption{Caption}
       % \label{tab:my_label}
    \end{table}
\noindent In each case, there exists an isolated point in $X_0(\ell)(\Q)$ with the indicated $j$-invariant.
\end{theorem}

\begin{proof}
Given $\im \rho_{E,\ell^{\infty}}$, we may apply a streamlined version of the algorithm of \cite{Lee25}. For the implementation, see \href{https://github.com/abbey-bourdon/rational-isolated-prime-power/blob/main/isolated_points_from_image_Gamma0.m}{\tt{isolated\_points\_from\_image\_Gamma0.m}}. The algorithm outline follows as in the previous theorem, with $X_0(n)$ in place of $X_1(n)$ in each relevant step. Being able to rule out isolated points on $X_0(\ell^a)$ when $\im \rho_{E,\ell^a}$ is of genus 0 follows from \cite[Theorem 48]{Lee25} or \cite[Theorem 1.3]{Terao}. As before, we ran this algorithm on the list of all known images as appearing in \href{https://github.com/abbey-bourdon/rational-isolated-prime-power/blob/main/elladicgens.txt}{\tt{elladicgens.txt}}. The \href{https://github.com/abbey-bourdon/rational-isolated-prime-power/blob/main/Gamma0results.out}{\tt{output}} is the empty set aside from the images associated to certain rational points on $X_0(\ell)$. The $j$-invariants and values of $\ell$ are given above. Moreover, each $j$-invariant gives rise to a point in $X_0(\ell)(\Q)$ in instances where this set is finite; see, for example, \cite[Table 4]{LR-TorsionFieldOfDefn}. Any such point is necessarily isolated. 
\end{proof}

\section{mod $\ell$ image in normalizer of non-split Cartan}

The classification of $\ell$-adic images for non-CM elliptic curves over $\Q$ is most incomplete in the case where the image of the mod $\ell$ Galois representation is contained in the normalizer of a nonsplit Cartan subgroup, denoted $C_{\text{ns}}^+(\ell)$. Here, work of Furio \cite[Theorem 1.9]{Furio2024} gives restrictions on $\im \rho_{E,\ell^{\infty}}$. These are sufficient to obtain the following result.
\begin{theorem}\label{CsnThm}
    Let $\ell$ be an odd prime and $n\in\Z^+$. Let $E/\Q$ be a non-CM elliptic curve. If $\im \rho_{E,\ell}$ is contained in the normalizer of a non-split Cartan subgroup, then there is no isolated point $x \in X_1(\ell^n)$ with $j(x)=j(E)$. 
\end{theorem}

\begin{proof}
By Theorem \ref{AlgThmGamma1}, we may assume $\im \rho_{E,\ell^{\infty}}$ is not a known image, so by \cite[Theorem 1.9]{Furio2024} it remains to address the following cases:
\begin{itemize}
    \item The group $\im \rho_{E,\ell^{\infty}}$ has level $\ell^d$ and up to conjugation $\im \rho_{E,\ell^d}=C_{\text{ns}}^+(\ell^d)$, the normalizer of a non-split Cartan subgroup of level $\ell^d$. Thus with an appropriate choice of basis, there is a quadratic extension $K/\Q$ such that
    \[
\rho_{E,\ell^d}(\Gal_K)=\left\lbrace \begin{bmatrix}
a& \epsilon b\\
b & a
\end{bmatrix} |\, a,b \in \Z/\ell^d\Z, \, (a,b) \neq (0,0) \pmod \ell\right\rbrace,
\] where $\epsilon$ is a fixed integer which is not a quadratic residue modulo $\ell$. The order of this group is $\ell^{2d-2}(\ell^2-1)$. A calculation shows that the only matrix in $\rho_{E,\ell^d}(\Gal_K)$ fixing a point $P \in E(\overline{\Q})$ of order $\ell^d$ is the identity matrix. It follows that $K(P)/K$ has order $\ell^{2d-2}(\ell^2-1)$. Hence $\Q(P)/\Q$ has degree at least $\ell^{2d-2}(\ell^2-1)$, and in fact equality must hold since this is always an upper bound. From this we may deduce that $x\in X_1(\ell^n)$ associated to $E$ has degree $\frac{1}{2}(\ell^{2n-2}(\ell^2-1))$, and so $\deg(x)=\deg(f(x)) \cdot \deg(x)$ where $f: X_1(\ell^n) \rightarrow X_1(1)$ is the natural map. By Theorem \ref{BELOV_level}, the point $f(x) \in X_1(1)$ is isolated, which is a contradiction.

\item The group $\im \rho_{E,\ell^{\infty}}$ has level $\ell^2$ and 
\[ \im \rho_{E,\ell^2} \cong C_{\text{ns}}^+(\ell) \ltimes \left\lbrace I + \ell\begin{bmatrix}
a& \epsilon b\\
-b & c
\end{bmatrix}\right \rbrace \] 
where $\epsilon$ is a fixed integer which is not a quadratic residue modulo $\ell$. 
By Theorem \ref{BELOV_level}, it suffices to show there are no isolated points associated to $E$ on $X_1(\ell^2)$ or $X_1(\ell)$. Since $\im \rho_{E,\ell} = C_{\text{ns}}^+(\ell)$, it follows from \cite[Lemma 5.2]{GJNajman20} that any point on $X_1(\ell)$ associated to $E$ has degree $(\ell^2-1)/2=\deg(X_1(\ell) \rightarrow X_1(1))$. Thus an isolated point on $X_1(\ell)$ associated to $E$ would map to an isolated point on $X_1(1)$ by Theorem \ref{BELOV_level}, and we have a contradiction. It remains to show there are no isolated points associated to $E$ on $X_1(\ell^2)$.

If $\ell>13$, then we may apply Proposition \ref{PropDegModCurve}. Indeed, in this case, any isolated point $x\in X_1(\ell^n)$ with $j(x)=j(E)$ would satisfy \[
\deg(x)=\deg(f(x))\cdot \deg(f),
\]
where $f: X_1(\ell^n) \rightarrow X_1(1)$ is the natural map. This would imply $f(x)\in X_1(1) \cong \mathbb{P}^1$ is isolated by Theorem \ref{BELOV_level}, which is a contradiction. So assume $\ell \leq 13$. For these values of $\ell$, the genus of $X_1(\ell^2)$ is given below:
    \begin{table}[H]
        \centering
        \begin{tabular}{c|c}
           $\ell=3$  &  0\\
           \hline
            $\ell=5$ & 12 \\
               \hline
            $\ell=7$ & 69 \\
            \hline
            $\ell=11$ & 526 \\
                           \hline
            $\ell=13$ & 1070 \\
        \end{tabular}
      %  \caption{Caption}
       % \label{tab:my_label}
    \end{table}
We will show that the degree of any closed point $x=[E,P]\in X_1(\ell^2)$ with $j(E)=j(x)$ is at least $\ell (\ell^2-1)/2$. Comparing this to the genus table given above, we observe that the degree of $x$ is greater than the genus of $X_1(\ell^2)$. Thus the Riemann-Roch space associated to $x$ has dimension at least 2, and so $x$ is not isolated.  
    
The group $\im\rho_{E,\ell^2}$ has order $2(\ell^2-1)\ell^3$ where the order of $C_{\text{ns}}^+(\ell)$ is $2(\ell^2-1)$. This implies that the degree of the field $\Q(E[\ell^2])$ over $\Q(E[\ell])$ is $\ell^3$.  Let $L \coloneqq\Q(E[\ell])$. Take $Q \in E(\overline{\Q})$ such that $\{P,Q\}$ is a basis for $E[\ell^2]$.
Since $[L(Q):L]\leq \ell^2$, we have
\[
\ell^3=[L(P,Q):L]\leq [L(Q,P):L(Q)]\cdot \ell^2.
\]
Hence
\[ \ell \leq [L(Q,P):L(Q)] \leq [L(P) : L] \leq [\Q(P):\Q(\ell P)]. \]
Since $[\Q(\ell P):\Q] =\ell^2-1$ by \cite[Lemma 5.2]{GJNajman20}, the inequalities above show $ \ell(\ell^2-1) \leq [\Q(P):\Q]$. Thus $x=[E,P]\in X_1(\ell^2)$ has degree greater than or equal to $\ell (\ell^2-1)/2$. \qedhere
\end{itemize}
\end{proof}

\begin{corollary}\label{CsnCor}
        Let $\ell$ be an odd prime and $n\in \Z^+$. Let $E/\Q$ be a non-CM elliptic curve. If $\im \rho_{E,\ell}$ is contained in the normalizer of a non-split Cartan subgroup, then there is no isolated point $x \in X_0(\ell^n)$ with $j(x)=j(E)$. 
\end{corollary}
\begin{proof}
    Let $x=[E, \langle P \rangle] \in X_0(\ell^n)$. By Theorem \ref{AlgThmGamma0}, we may assume $\im \rho_{E,\ell^{\infty}}$ is not a known image, so it remains to address the following cases as in the proof of the previous theorem:
\begin{itemize}
    \item The group $\im \rho_{E,\ell^{\infty}}$ has level $\ell^d$ and $\im \rho_{E,\ell^d}=C_{\text{ns}}^+(\ell^d)$, up to conjugation. In the previous theorem, we established that $[E,P]\in X_1(\ell^n)$ has degree $\frac{1}{2}(\ell^{2n-2}(\ell^2-1))$. Thus $\deg(x)=\ell^{n-1}(\ell+1)$. By Theorem \ref{BELOV_level}, the point $f(x) \in X_0(1)$ is isolated, where $f:X_0(\ell^n) \rightarrow X_0(1)$ is the natural map. This is a contradiction.

\item The group $\im \rho_{E,\ell^{\infty}}$ has level $\ell^2$ and 
\[ \im \rho_{E,\ell^2} \cong C_{\text{ns}}^+(\ell) \ltimes \left\lbrace I + \ell\begin{bmatrix}
a& \epsilon b\\
-b & c
\end{bmatrix}\right \rbrace. \] 
As in the previous theorem, we can reduce to the case of analyzing isolated points on $X_0(\ell^2)$. The proof of the previous theorem shows that the degree of $[E,P]\in X_1(\ell^2)$ is at least $\ell (\ell^2-1)/2$.  Hence $\deg(x) \geq \ell+1$.

If $\ell>13$, then we may apply Corollary \ref{CorModCurve} and as before use Theorem \ref{BELOV_level} to show no isolated points exist on $X_0(\ell^n)$. So assume $\ell \leq 13$. For these values of $\ell$, the genus of $X_0(\ell^2)$ is given below:
    \begin{table}[H]
        \centering
        \begin{tabular}{c|c}
           $\ell=3$  &  0\\
           \hline
            $\ell=5$ & 0 \\
               \hline
            $\ell=7$ & 1 \\
            \hline
            $\ell=11$ & 6 \\
                           \hline
            $\ell=13$ & 8 \\
        \end{tabular}
      %  \caption{Caption}
       % \label{tab:my_label}
    \end{table}
Again we observe that the degree of $x$ is greater than the genus of $X_0(\ell^2)$ and an application of the  Riemann-Roch Theorem shows $x$ it is not isolated. \qedhere
\end{itemize}
\end{proof}
\section{Proof of Theorem \ref{X1ellThm}}
The claim about CM $j$-invariants follows from \cite[Theorem 7.1]{BELOV}, so we may henceforth restrict to the non-CM case. Let $E/\Q$ be a non-CM elliptic curve. If $\im \rho_{E,\ell^{\infty}}$ is conjugate to one of the known subgroups appearing in \cite{2adicImage,RSZB2022}, then the result follows from Theorem \ref{AlgThmGamma1}. Since the 2-adic image classification is complete, we may now assume $\ell$ is odd, and we may further assume $\im \rho_{E, \ell^{\infty}}$ is not one of the subgroups known to occur. By Theorem \ref{CsnThm}, we may assume $\im \rho_{E,\ell}$ is not contained in the normalizer of a non-split Cartan subgroup. Thus it follows from \cite[Theorem 1.6]{RSZB2022} that $\ell=7$ and $\im \rho_{E,7^{\infty}}$ is contained in 49.196.9.1. 

So suppose $\im \rho_{E,7^{\infty}}$ is contained in 49.196.9.1. Then $\im \rho_{E,7}$ must equal 7Ns.2.1, 7Ns.3.1, or 7Ns. By \cite[Theorem 1.5]{ZywinaImages} first two groups\footnote{In this paper, the groups are notated $H_{1,1}$ and $G_1$, respectively.} occur if and only if $j(E)=3^3\cdot5\cdot7^5/2^7$. Since degrees of closed points do not depend on the chosen equation for $E$, we may assume $E$ is given by LMFDB label \href{https://www.lmfdb.org/EllipticCurve/Q/2450/y/1}{2450.y1}. But in this case, $\im \rho_{E,7^{\infty}}=7.112.1.2$, and so any isolated point on $X_1(7^n)$ associated to $E$ will map to an isolated point on $X_1(7)$ via the natural projection map by Theorem \ref{BELOV_level}. As $X_1(7)$ has genus 0, this is a contradiction. Thus we may assume $\im \rho_{E,7}=$ 7Ns.

Magma computations referred to in the remainder of the proof can be found \href{https://github.com/abbey-bourdon/rational-isolated-prime-power/blob/main/Image49.196.9.1.txt}{here}. Since $\det(\im \rho_{E,49})=(\Z/49\Z)^{\times}$, a Magma computation shows that $\im \rho_{E,49}$ is either all of 49.196.9.1, or else it belongs to the unique conjugacy class of subgroups of index $49$. In the second case, by choosing a basis appropriately, we may assume \[\im \rho_{E,49}=\left\langle \left[\begin{matrix}
    1 & 0\\ 37 & 48
\end{matrix}\right], \left[\begin{matrix}
    20 & 4\\ 18 & 21
\end{matrix}\right], \left[\begin{matrix}
    31 & 17\\ 3 & 23
\end{matrix}\right], \left[\begin{matrix}
    22 & 0\\ 0 & 22
\end{matrix}\right], \left[\begin{matrix}
    34 & 26\\ 19 & 16
\end{matrix}\right], \left[\begin{matrix}
    33 & 26\\ 19 & 15
\end{matrix}\right] \right\rangle \leq \GL_2(\Z/49\Z).
\] We will rule out this possibility by following the proof of \cite[Theorem 1]{Vukorepa2022}. A computation shows this subgroup is conjugate to an index 7 subgroup of the normalizer of a split Cartan subgroup mod 49, denoted $C_{\text{s}}^+(49)$. Elliptic curves over $\Q$ with mod 49 image landing in $C_{\text{s}}^+(49)$ give rise to rational points on the modular curve $X_s^+(49)$; see \cite[Proposition 3.3]{ZywinaImages} for details. There is a $\Q$-isomorphism 
\[
X_s^{+}(7^2) \cong X_0^+(7^4),
\]
where the latter is the quotient of $X_0(7^4)$ by the Atkin-Lehner involution. Moreover, the moduli interpretation of this isomorphism guarantees cusps map to cusps and CM points map to CM points; see \cite[Section 2]{BiluParentRebolledo}. Since there are no non-CM, non-cuspidal rational points on $X_0^+(7^4)$ by \cite[Theorem 3.14]{MomoseShimura}, there are no non-CM, non-cuspidal rational points on $X_s(7^2)$. Thus we may assume $\im \rho_{E,49}=49.196.9.1$. 

In considering possible mod 343 images with surjective determinant, we find that only the complete pre-image of 49.196.9.1 reduces to exactly 49.196.9.1 mod 49. Thus $\im \rho_{E,7^{\infty}}=49.196.9.1$ by, for example, \cite[Proposition 3.5]{BELOV}. We can apply \href{https://github.com/abbey-bourdon/rational-isolated-prime-power/blob/main/isolated_points_from_image_Gamma1.m}{\tt{isolated\_points\_from\_image\_Gamma1.m}} as in $\S3.1$ to show there are no isolated points on $X_1(7^n)$ associated to elliptic curves over $\Q$ with this 7-adic image.

\begin{corollary}\label{7adicCor}
    Suppose $E/\Q$ is a non-CM elliptic curve with $\im \rho_{E,7^{\infty}}$ contained in $49.196.9.1$. Then  $\im \rho_{E,7^{\infty}}=49.196.9.1$.
\end{corollary}

\begin{proof}
    The argument follows from the proof of the previous theorem, combined with the observation that an elliptic curve $E/\Q$ with $j(E)=3^3\cdot5\cdot7^5/2^7$ cannot have $\im \rho_{E,7^{\infty}}$ contained in 49.196.9.1. Indeed, there exists a quadratic extension $K$ such that $E/K$ is isomorphic to the base extension to $K$ of the elliptic curve $E'$ with LMFDB label \href{https://www.lmfdb.org/EllipticCurve/Q/2450/y/1}{2450.y1}. Since $K(E'[49])/K(E'[7])$ has degree $7^4$, it follows that $\Q(E[49])/\Q(E[7])$ has order $7^4$. Thus, $\im \rho_{E,7^{\infty}}$ has level 7 by \cite[Proposition 3.5]{BELOV}, and it cannot be contained in 49.196.9.1.
\end{proof}

\section{Proof of Theorem \ref{X0ellThm}}

If $E/\Q$ is a CM elliptic curve, then for each prime $\ell$ which is split in the CM field there exists a quadratic point on $X_0(\ell^n)$ associated to $E$ for all $n \in \Z^+$; such a point is sporadic (and hence isolated) for $\ell^n>131$ by work of Harris and Silverman \cite{harrissilverman}, Ogg \cite{oggpaper}, and Bars \cite{bars}. So we may assume $E/\Q$ is non-CM. As in the proof of Theorem \ref{X1ellThm}, we may assume $\ell$ is odd and that $\im \rho_{E, \ell^{\infty}}$ is not one of the subgroups known to occur by Theorem \ref{AlgThmGamma0}. By Corollary \ref{CsnCor}, we may assume that $\im \rho_{E,\ell}$ is not contained in the normalizer of a non-split Cartan subgroup. Thus, it follows from \cite[Theorem 1.1.6]{RSZB2022} and Corollary \ref{7adicCor} that $\ell=7$ and $\im \rho_{E,7^{\infty}}=49.196.9.1$. In this case, an application of \href{https://github.com/abbey-bourdon/rational-isolated-prime-power/blob/main/isolated_points_from_image_Gamma0.m}{\tt{isolated\_points\_from\_image\_Gamma0.m}} as in Section 3.2 shows that there are no isolated points on $X_0(\ell^n)$ associated to $E/\Q$ with this 7-adic image.

\bibliographystyle{amsplain}
\bibliography{bibliography}
\end{document}